\documentclass[]{article}
\usepackage{graphicx}
\usepackage{graphics}
\usepackage{amssymb,amsmath,amsthm,epsfig}
\newtheorem{theorem}{Theorem}
\newtheorem{lemma}{Lemma}
\newtheorem{corollary}{Corollary}
\newtheorem{definition}{Definition}
\newtheorem{proposition}{Proposition}

\newcommand{\vp}{\varphi^{(\alpha)}_{n,c}}

\newcounter{reh}
\begin{document}

\begin{center}
{\large {\bf  The finite Hankel transform operator: Some explicit and local estimates
of the eigenfunctions and eigenvalues decay rates.}}

\vskip 0.5cm Mourad Boulsane and Abderrazek Karoui {\footnote{
Corresponding author: Abderrazek Karoui, Email: abderrazek.karoui@fsb.rnu.tn
This work was supported by the DGRST research Grant UR13ES47 and the project CMCU PHC Utique 15G1504.}}
\end{center}
\vskip 0.25cm {\small
University of Carthage,
Department of Mathematics, Faculty of Sciences of Bizerte, Jarzouna, 7021, Tunisia.}\\

\noindent{\bf Abstract}--- For fixed real numbers $c>0,$ $\alpha>-\frac{1}{2},$ the finite Hankel transform operator, denoted
by $\mathcal{H}_c^{\alpha}$ is given by  the  integral operator defined
on $L^2(0,1)$ with kernel $K_{\alpha}(x,y)= \sqrt{c xy} J_{\alpha}(cxy).$  To the operator $\mathcal{H}_c^{\alpha},$
we associate a positive, self-adjoint compact integral operator $\mathcal Q_c^{\alpha}=c\, \mathcal{H}_c^{\alpha}\, \mathcal{H}_c^{\alpha}.$ Note that the integral operators $\mathcal{H}_c^{\alpha}$ and $\mathcal Q_c^{\alpha}$ commute with a Sturm-Liouville differential operator $\mathcal D_c^{\alpha}.$ 
 In this paper, we first give some useful estimates and bounds of the eigenfunctions $\vp$ of $\mathcal H_c^{\alpha}$ or $\mathcal Q_c^{\alpha}.$ These estimates and bounds are obtained by using some special techniques from the theory of Sturm-Liouville operators, that we apply to the differential operator $\mathcal D_c^{\alpha}.$  If 
 $(\mu_{n,\alpha}(c))_n$ and $\lambda_{n,\alpha}(c)=c\, |\mu_{n,\alpha}(c)|^2$ denote the infinite and countable sequence of the eigenvalues of the operators $\mathcal{H}_c^{(\alpha)}$ and $\mathcal Q_c^{\alpha},$ arranged in the decreasing order of their magnitude, then we show an unexpected result that for a given integer $n\geq 0,$ $\lambda_{n,\alpha}(c)$ is decreasing with respect to the parameter $\alpha.$
As a consequence, we show that for $\alpha\geq \frac{1}{2},$ the $\lambda_{n,\alpha}(c)$ and the $\mu_{n,\alpha}(c)$ have a super-exponential decay rate. Also, we give a lower decay rate of these eigenvalues. As it will be seen, the previous results are essential tools for the analysis of a spectral approximation scheme based on the eigenfunctions of the finite Hankel transform operator. Some numerical examples will be provided to illustrate the results of this work.\\

\noindent {{\bf 2010 Mathematics Subject Classification.}} Primary
42C10, 65L70. Secondary 41A60, 65L15.\\

\noindent {\it  Key words and phrases.} Finite Hankel transform operator, Sturm-Liouville operator,  eigenfunctions and eigenvalues, prolate spheroidal wave functions, approximation of Hankel band-limited functions.\\

\section{Introduction}

We first recall that for a bandwidth $c>0$ and a real number $\alpha >-1/2,$ the circular prolate spheroidal wave functions (CPSWFs),
denoted by $(\vp)_{n\geq 0}$  are the different eigenfunctions of the following finite Hankel transform operator, see for example \cite{KM2, Slepian3}
\begin{equation}\label{finite_Hankel}
\mathcal{H}_c^{\alpha} (f)(x)=\int_0^1 \sqrt{ctx} J_{\alpha}(ct x) f(t)\, dt,\quad x\in [0,1],
\end{equation}
that is
\begin{equation}\label{integral_eq1}
\mathcal{H}_c^{\alpha} \vp(x)\, dy =\int_0^1 \sqrt{cxy} J_{\alpha}(cxy) \vp(y)\, dy= \mu_{n,\alpha}(c)\vp(x),\quad n\geq 0,\quad x\in [0,1].
\end{equation}
Here, $J_{\alpha}$ is the Bessel function of the first type and order $\alpha >-1/2.$
We recall that the Hankel transform is defined on $L^2(0,\infty)$ by
\begin{equation}\label{Hankel_transform}
\mathcal H^{\alpha} (f)(x)= \int_0^{+\infty} \sqrt{xy} J_{\alpha}(xy) f(y)\, dy.
\end{equation}
Moreover, for
$c>0,$ the Hankel Paley-Wiener space is the space of functions from $L^2(0,\infty),$ having compactly supported Hankel transforms, that is
\begin{equation}
\label{Paley_space}
\mathcal B_c^{\alpha}=\{ f\in L^2(0,\infty);\,\ \mbox{Supp} \mathcal H^{\alpha}(f)\subset [0,c]\}.
\end{equation}

Although,  in the literature, there exist  extensive works devoted to the numerical computation
of the CPSWFs and their associated eigenvalues $\mu_{n,\alpha}(c),$ see for example \cite{Tatiana1,KM2,Tatiana2,Slepian3}, the subject of the explicit estimates and bounds of the $\vp,$ as well as the decay rate of the eigenvalues  $\mu_{n,\alpha}(c)$ or $\lambda_{n,\alpha}(c),$ is still unexplored. Our objective from this work is to
provide the reader with some useful explicit  local estimates and bounds of the CPSWFs, as well as some
explicit lower and upper bounds of the eigenvalues $\mu_{n,\alpha}(c).$\\

We first mention that the interest from the study of the eigenfunctions of the finite Hankel transform (CPSWFs)  and in general the prolate spheroidal wave functions, comes from the fact they are widely used in various scientific area, such as applied mathematics, physics, engineering,
see \cite{Hogan} for some of these  concrete applications.

In the pioneer work \cite{Slepian3}, D. Slepian has shown that the compact integral operator $\mathcal{H}_c^{\alpha}$ commutes with the following differential operator $\mathcal D^{\alpha}_c$ defined on $C^2([0,1])$ by
\begin{equation}
\label{differ_operator1}
\mathcal D^{\alpha}_c (\phi)(x)  = \dfrac{d}{dx} \left[ (1-x^2)\dfrac{d}{dx} \phi(x) \right] + \left( \dfrac{\dfrac{1}{4}-\alpha^2}{x^2}-c^2x^2 \right)\phi(x).
\end{equation}
Hence,  $\vp$ is the $n-$th order bounded   eigenfunction of the operator $-\mathcal D^{\alpha}_c,$  associated with the eigenvalue $\chi_{n,\alpha}(c),$ that is
\begin{equation}
\label{differ_operator2}
-\dfrac{d}{dx} \left[ (1-x^2)\dfrac{d}{dx} \vp(x) \right] - \left( \dfrac{\dfrac{1}{4}-\alpha^2}{x^2}-c^2x^2 \right)\vp(x)=\chi_{n,\alpha}(c)\vp(x),\quad x\in [0,1].
\end{equation}

In this work, we  take advantage from the commutativity property of the operators $\mathcal D^{\alpha}_c$ and $\mathcal{H}_c^{\alpha}$ and prove some useful local estimates and bounds of the $\vp(x),\, x\in I.$
Note that some estimates and bounds of the classical prolate spheroidal wave functions and their generalized versions, were already given in \cite{Bonami-Karoui1, Karoui-Souabni}. Nonetheless, in our present case of the CPSWFs, the techniques used in the previous references have to be modified and combined with new techniques based on the use of the Sturm-Liouville comparison theorem and Butlewski's theorem. These new techniques are needed in order to handle 
the extra difficulty caused by the  singularity at $x=0,$ appearing in the  differential operator $\mathcal D^{\alpha}_c.$  Also, by using the characterization of the eigenvalues $\lambda_{n,\alpha}(c)$ in terms of an energy maximization problem, combined with Griffith's theorem which is a Paley-Wiener theorem for the Hankel transform,  we prove an interesting result that the 
$\lambda_{n,\alpha}(c)$ are decreasing with respect to the parameter $\alpha.$ As a consequence, and by using the sharp decay rate of the eigenvalues of the finite Fourier transform, given in \cite{Bonami-Karoui2}, we give a super-exponential decay rate of the $\lambda_{n,\alpha}(c),$ for $\alpha\geq \frac{1}{2}.$\\

This work is organised as follows. In section 2, we give some mathematical preliminaries related to the properties and computation of the CPSWFs.  In section 3, we first provide a local estimate for the $\vp(x).$  Then, we give a bound of
$|\vp(x)|$ for $x\in [0,1].$ The previous results are obtained under the condition that $\chi_{n,\alpha}(c)> c^2+\alpha^2-\frac{1}{4},$ where $\chi_{n,\alpha}(c)$ is  the $n-$th eigenvalues 
$\chi_{n,\alpha}(c)$ of the differential operator $\mathcal Q_c^{\alpha}.$ By using the classical Strurm-Liouville comparison theorem, we prove that $\chi_{n,\alpha}(c)$ passes through $c^2+\alpha^2-\frac{1}{4}$ when $n$ is around 
$\frac{c}{\pi}.$ 
In section 4, we give an upper and a lower  bound of the super-exponential decay rate of the eigenvalues $\lambda_{n,\alpha}(c).$  Finally, in section 5, we provide the reader with some numerical examples  that illustrate the different results of this work. Moreover, in this section, we also  show that the $\vp$ are well adapted for the approximation of Hankel band-limited and almost band-limited functions.

\section{Mathematical preliminaries.}

In this section, we first give a brief description of the computation and the decay rate of the
 series expansion coefficients $d_k^n$ of  the eigenfunctions $\vp$ in an appropriate  basis of $L^2([0,1]).$ This basis is
 given by the orthogonal functions  $\widetilde T_{k,\alpha}(x)= x^{\alpha+1/2} (-1)^k \sqrt{2(2k+\alpha+1)}
P_k^{(\alpha,0)}(1-2x^2),\, k\geq 0.$ Here,  $P_k^{(\alpha,\beta)}$ is the Jacobi polynomial of degree $k$ and parameters $\alpha, \beta >-1,$ normalized so that
${ P_k^{(\alpha,\beta)}(1)={n+\alpha \choose n}=\frac{\Gamma(n+\alpha+1)}{\Gamma(\alpha+1)\Gamma(n+1)}.}$
Then, we relate the eigenvalues of the compact and positive operator
\begin{equation}
\label{Integral_Operator2}
\mathcal Q_c^{\alpha} = c \mathcal H_c^{\alpha} (\mathcal H_c^{\alpha})^*=c \mathcal H_c^{\alpha} (\mathcal H_c^{\alpha})
\end{equation}
to the solutions of a classical energy maximization problem over the Paley-Wiener space $\mathcal B_c^{\alpha},$ given by
\eqref{Paley_space}.

Note that thanks to the  important  commutativity property of the differential and integral  operators $\mathcal D^{\alpha}_c$ and $\mathcal H_c^{\alpha},$ D. Slepian has developed in \cite{Slepian3}, an efficient computational scheme for the $\vp(x),\, x\geq 0,$ as well as for their corresponding eigenvalues
$\chi_{n,\alpha}(c)$ and $\mu_{n,\alpha}(c).$  The Slepian scheme for the computation of
$\vp(x),$ is given by the following series expansion,
 \begin{equation}\label{expansion1}
\vp(x)=\sum_{k=0}^{+ \infty} d_{k}^n \; \widetilde T_{k,\alpha}(x),\quad \widetilde T_{k,\alpha}(x)=
(-1)^k \sqrt{2(2k+\alpha+1)} x^{\alpha+\frac{1}{2}} P_k^{(\alpha,0)}(1-2 x^2),\quad x\in [0,1].
 \end{equation}
Here, $d_{k}^n$ are the expansion coefficients, given as the eigenvectors a tri-diagonal infinite order matrix.  Moreover, by combining  the integral equation \eqref{integral_eq1} and the previous expansion, D. Slepian has given the following analytic extension of the  $\vp,$ over the unbounded interval $[1,+\infty),$
\begin{equation}\label{expansion2}
\vp(x)=\frac{1}{\mu_{n,\alpha}(c)}\sum_{k\geq 0}
(-1)^k d_{k}^{n}\sqrt{2(2k+\alpha+1)} \frac{J_{2k+\alpha+1}(cx)}{\sqrt{cx}},\quad
x\geq 1.
\end{equation}
By evaluating  the two expansions \eqref{expansion1} and \eqref{expansion2}  at $x=1,$ one gets the following expression of the eigenvalues $\mu_{n,\alpha}(c),$
\begin{equation}\label{Eigenvals1}
\mu_{n,\alpha}(c)=\frac{1}{\sqrt{c}}\frac{\sum_{k\geq 0}(-1)^k
d_{k}^{n}\sqrt{2(2k+\alpha+1)} J_{2k+\alpha+1}(c)}{\sum_{k=0}^{+ \infty} d_{k}^n \; \widetilde T_{k,\alpha}(1)}.
\end{equation}

In this work, we  check that for a fixed positive integer $n,$ the sequence
$(d_k^n)_{k\geq 0}$ has a super-exponential decay rate. Consequently, the previous formulae for computing the $\vp$ and their eigenvalues $\mu_{n,\alpha}(c)$ are practical and highly accurate. Note that by using
 a slightly modified techniques of those used in \cite{Slepian3}, one can easily check that  the expansion coefficients can be computed by solving  the following tri-diagonal system
\begin{eqnarray}\label{Eq1.14}
\sum_{k\geq 0}\left[  c^{2} a_{k,\alpha} d_{k-1}^n +\left(( 2k+\alpha+\dfrac{1}{2})( 2k+\alpha+\dfrac{3}{2})+c^{2} b_{k,\alpha}\right) \, d_k^n+ c^{2} a_{k+1,\alpha} d_{k+1}^n\right]\widetilde T _{k,\alpha}(x)&&  \nonumber \\
\qquad =  \chi_{n,\alpha}(c) \sum_{k\geq 0} d_{k}^n \; \widetilde T_{k,\alpha},
\end{eqnarray}
where
\begin{equation}\label{Eq1.12}
a_{k+1,\alpha}=\dfrac{(k+1)(\alpha+k+1)}{(\alpha+2k+2) \sqrt{\alpha+2k+1}\sqrt{\alpha+2k+3})},\quad
b_{k,\alpha}=\dfrac{1}{2} \left(\dfrac{\alpha^2}{(\alpha+2k+2)(\alpha+2k)}+1 \right).
\end{equation}
The previous system can be written in the following eigensystem
\begin{equation}
\label{eigensystem}
 M  D  = \chi_n   D,\quad M=[m_{i,j}]_{i,j\geq 0},\quad D= [d_k^n]^T_{k\geq 0}
\end{equation}
with  $m_{k,j}=0,\,\, \mbox{ if } \mid{k-j}\mid \geq 2$ and
\begin{equation}\label{1.23}
m_{k,k-1}=c^{2} a_{k,\alpha},\quad m_{k,k}= (\alpha+2k+\dfrac{1}{2})(\alpha+2k+\dfrac{3}{2})+c^{2} b_{k,\alpha},\quad
m_{k,k+1}=m_{k,k+1}.
\end{equation}
Also, note that the expansion coefficients
$(d_{k}^n)_k$ are related to $\vp$ by the following relation,
\begin{equation}\label{Eq1.9}
 d_k^n=\dfrac{(-1)^k\sqrt{2(2k+\alpha+1)}}{\mu_{n,\alpha}(c)}  \int_{0}^{1}{\vp(y)  \dfrac{J_{2k+\alpha+1}(y)}{\sqrt{cy}} dy}.
\end{equation}
In fact, from \eqref{integral_eq1}, we have
\begin{eqnarray*}
d_{k}^n &=& \dfrac{1}{\mu_{n,\alpha}(c)}\int_{0}^{1}{ \widetilde T_{k,\alpha}(x) \; \int_{0}^{1}{\sqrt{cxy} J_{\alpha}(cxy) \vp(y)dy}  dx}\\
&=& \dfrac{1}{\mu_{n,\alpha}(c)}\int_{0}^{1}{\vp(y)  \; \int_{0}^{1}{\sqrt{cxy} J_{\alpha}(cxy) \widetilde {T_{k,\alpha}}(x)}  dx dy}
\end{eqnarray*}
Also, from \cite{KM2, Slepian3}, one has
\begin{equation}\label{Eq1.8}
\int_{0}^{1}{\sqrt{cxy} J_{\alpha}(cxy) \widetilde T_{k,\alpha}(x)dx}= (-1)^k \sqrt{2(2k+\alpha+1)} \dfrac{J_{2k+\alpha+1}(y)}{\sqrt{cy}}.
\end{equation}
By combining the previous two equalities, one gets \eqref{Eq1.9}.

It is well known that the eigenvalues $\chi_{n,\alpha}(c)$ of the differential operator $\mathcal D^{\alpha}_c$ satisfy the differential equation
 \begin{equation}\label{differ_chi}
 \dfrac{\partial\chi_{n,\alpha}(c)}{\partial c}= 2c \int_{0}^{1} x^2\left(\vp(x)\right)^2\, dx.
  \end{equation}
Since $\chi_{n,\alpha}(0)=(2n+\alpha+\dfrac{1}{2})(2n+\alpha+\dfrac{3}{2})$ is the $n-$th eigenvalue of the differential operator $\mathcal D^{\alpha}_0,$ and
since $0\leq c^2x^2 \leq c^2$ for all $x\in [0,1]$ , then by using the Min-Max principle,
the $n-$th eigenvalue  $\chi_{n,\alpha}(c)$ of the differential operator $\mathcal D^{\alpha}_c$ satisfies the following bounds,
\begin{equation}\label{bounds_chi}
\left(2n+\alpha+\dfrac{1}{2}\right)\left(2n+\alpha+\dfrac{3}{2}\right) \leq \chi_{n,\alpha}(c)  \leq \left(2n+\alpha+\dfrac{1}{2}\right)\left(2n+\alpha+\dfrac{3}{2}\right)+ c^2.
\end{equation}

Next, we briefly check a classical result  that the eigenvalues $\lambda_{n,\alpha}(c)$ of the integral operator $\mathcal Q_c^{\alpha},$ given by \eqref{Integral_Operator2} are characterized as the solutions of an energy maximization problem over the Hankel Paley-Wiener space  $\mathcal B_c^{\alpha}.$ In fact, from [\cite{Watson}, p.154], we have
\begin{equation}\label{kernel}
G_{\alpha}(x,y)=\int_0^1 \sqrt{xy} J_{\alpha}(xt)J_{\alpha}(yt) t \, dt = \frac{\sqrt{xy}}{x^2-y^2}\left(x J_{\alpha+1}(x) J_{\alpha}(y)-y J_{\alpha+1}(y) J_{\alpha}(x)\right).
\end{equation}
On the other hand, by using the previous identity and since $\mathcal H_c^{\alpha}$ is self-adjoint, then a straightforward computation, gives us
\begin{equation}
\mathcal Q_c^{\alpha}(f)(x)= c\, \mathcal H_c^{\alpha} \mathcal H_c^{\alpha}(f)(x)= c\int_0^1 G_{\alpha}(cx, cy) f(y)\, dy,
\end{equation}
where the kernel $G_{\alpha}(x,y)$ is given by \eqref{kernel}.  On the other hand, since the Hankel transform operator is its own inverse and since by Plancherel formula, we have for $f\in \mathcal B_c^{\alpha},$
$${\displaystyle
\|f\|^2_{L^2(0,\infty)}=\|\mathcal H^{\alpha}f\|^2_{L^2(0,\infty)}=c \int_0^1 (\mathcal H^{\alpha}(f))^2(c x)\, dx,}$$
then, for $f\in \mathcal B_c^{\alpha},$ we have
\begin{eqnarray*}
\frac{\|f\|^2_{L^2(0,1)}}{\|f\|^2_{L^2(0,\infty)}}&=&\frac{\int_0^1 (f(t))^2\, dt}{\|f\|^2_{L^2(0,\infty)}}= \frac{\int_0^1 \left(\int_0^c \sqrt{X t} J_{\alpha}(X t)\mathcal H^{\alpha}(f)(X)\, dX\right) \cdot \left(\int_0^c \sqrt{Y t} J_{\alpha}(Yt)\mathcal H^{\alpha}(f)(Y)\, dY\right)\, dt}{c\int_0^1 (\mathcal H^{\alpha}(f))^2(c x)\, dx} \\
&=& \frac{c^2 \int_0^1 \left(\int_0^1 G_{\alpha}(cx, cy) \mathcal H^{\alpha}(f)(cy)\, dy\right)\cdot \mathcal H^{\alpha}(f)(cx)\, dx}{c\int_0^1 (\mathcal H^{\alpha}(f))^2(c x)\, dx}\\
&=&\frac{\int_0^1 \mathcal Q_c^{\alpha}      \mathcal H^{\alpha}(f)(cx) \cdot \mathcal H^{\alpha}(f)(cx)\, dx}{\int_0^1 (\mathcal H^{\alpha}(f))^2(c x)\, dx}.
\end{eqnarray*}
A standard result about the maximization of a quadratic form tells us that the solution of the energy maximization problem
\begin{equation}\label{energy_problem}
\mbox{ find } f = \arg \max_{f\in \mathcal B_c^{\alpha}} \frac{\|f\|^2_{L^2(0,1)}}{\|f\|^2_{L^2(0,\infty)}}
\end{equation}
is given by the first eigenfunction, with the largest eigenvalue $\lambda_{0,\alpha}(c)$ of the operator
$\mathcal Q_c^{\alpha}.$ Since   $\mathcal Q_c^{\alpha}=c \,\mathcal H_c^{\alpha} \mathcal H_c^{\alpha},$ then the eigenfunctions  of $\mathcal Q_c^{\alpha}$ are also the eigenfunctions $\vp$ of $\mathcal H_c^{\alpha}$ and the eigenvalues of $\mathcal Q_c^{\alpha}$ are related to the eigenvalues of $\mathcal H_c^{\alpha}$
by the following rule
\begin{equation}\label{Eigenvalues_relation}
\lambda_{n,\alpha}(c) = c\, |\mu_{n,\alpha}(c)|^2,\quad n\geq 0.
\end{equation}

Finally, we should mention that throughout this work, the eigenfunctions $\vp$ are normalized by the following rule,
\begin{equation}
\label{normal}
\int_0^1 \left(\vp(x)\right)^2\, dx = 1,\quad \int_0^{\infty} \left(\vp(x)\right)^2\, dx = \frac{1}{{\lambda_{n,\alpha}(c)}}.
\end{equation}

\section{Some explicit estimates and bounds of the eigenfunctions.}

In this paragraph, we give an explicit upper bound of  $|\vp(x)|$ with  $x\in I=[0,1],$ and $\alpha > -1/2.$ To this end,  we first   show that under some conditions on $n,$ $c,$  the maximum of $|\vp(x)|$ is  attained at $x=1.$ This is given by the following lemma.

\begin{lemma}\label{Sup_phi} Let $c>0,$ $\alpha > -1/2$ be two real numbers. If  $c^2>\alpha^2-\frac{1}{4},$ and $\chi_{n,\alpha}(c)>c^2+\alpha^2-\frac{1}{4},$ then we have
\begin{equation}\label{Ineq2}
\sup_{x\in [a_{\alpha},1]}|\vp(x)|= |\vp(1)|,\quad a_{\alpha}=\left\{\begin{array}{ll} 0&\mbox{ if } \alpha^2\leq 1/4\\
\left(\frac{\alpha^2-1/4}{c^2}\right)^{1/4}&\mbox{ if } \alpha>1/2.\end{array}\right.
\end{equation}
\end{lemma}

\noindent
{\bf Proof:}
We  recall that $\vp$ is a solution of the following differential equation
$$
 \dfrac{d}{dt} \left[p(t)(\vp)'(t)\right]+q_{\alpha}(t)\vp(t)=0, $$  with
 \begin{equation}
 \label{auxfct}
  p(t)=(1-t^2),   \qquad
  q_{\alpha}(t)= \chi_{n,\alpha}(c)-c^2 t^2+\dfrac{\dfrac{1}{4}-\alpha^2}{t^2}.
\end{equation}
Next, consider the auxiliary function $Q_{\alpha},$ defined on $[a_{\alpha},1]$ by
\begin{equation}
\label{auxfct1}
Q_{\alpha}(t)=\left(\vp(t)\right)^2+\frac{p(t)}{q_{\alpha}(t)}\left((\vp)'(t)\right)^2.
\end{equation}
By computing the derivative of $Q_{\alpha}$ and then  using the identity
$$(\vp)''(t)=\frac{2t}{1-t^2}(\vp)'(t)-\frac{1}{t^2}\left(\chi_n-c^2 t^2+\frac{1/4-\alpha^2}{t^2}\right) \vp(t),$$
one gets
\begin{equation}
\label{derivQ}
Q_{\alpha}'(t)=\frac{2t}{q_{\alpha}^2(t)} \left(\chi_{n,\alpha}(c)+c^2-2 c^2 t^2 -\frac{\alpha^2-1/4}{t^4}\right) \big((\vp)'(t)\big)^2.
\end{equation}
Note that
\begin{eqnarray}
\label{Ineq3}
\chi_{n,\alpha}(c)+c^2-2 c^2 t^2 -\frac{\alpha^2-1/4}{t^4}&\geq& 2 c^2+(\alpha^2-1/4)-2c^2 t^2-\frac{\alpha^2-1/4}{t^4}\nonumber \\
&\geq & (1-t^2)\left(2c^2-(\alpha^2-1/4)\frac{1+t^2}{t^4}\right)\geq 0,\quad t\in (a_{\alpha},1].
\end{eqnarray}
Hence, by combining \eqref{derivQ} and \eqref{Ineq3}, one concludes that
 $Q_{\alpha}$ is increasing on $[0,1]$ and consequently,
$$(\vp)^2(t) \leq Q_{\alpha}(t)\leq (\vp)^2(1),\quad t\in [a_{\alpha},1],$$
which concludes the proof of the lemma.

The following lemma provides us with a useful local estimate of the eigenfunctions $\vp.$
\begin{lemma}\label{localestimate}
Under the notation and conditions of the previous lemma, we have for $ \alpha > - 1/2,$
\begin{equation}\label{local_estimate}
\sup_{t\in [a_{\alpha},1]} \sqrt{1-t^2} |\vp(t)|\leq \sqrt{2}.
\end{equation}
\end{lemma}

\noindent
{\bf Proof:} We first consider the auxiliary function $K_{\alpha}(\cdot),$ defined  by
$K_\alpha(t)= (1-t^2) Q_{\alpha}(t),$
where $Q_{\alpha}$ is given by \eqref{auxfct1}.  Straightforward computations give us
\begin{eqnarray*}
K'_\alpha(t) &=& -2 t Q_{\alpha}(t)+(1-t^2)Q'_{\alpha}(t)\\
&=& -2 t (\vp)^2(t) +\frac{2t (1-t^2)}{q_{\alpha}(t)}\left(\frac{\chi_{n,\alpha}(c)+c^2-2c^2 t^2+(1/4-\alpha^2)/t^4}{q_{\alpha}(t)} -1\right) ((\vp)'(t))^2\\
&=& -2 t (\vp)^2(t) +\frac{2t (1-t^2)^2}{q_{\alpha}^2(t)}\left(c^2-\frac{\alpha^2-1/4}{t^4}\right) ((\vp)'(t))^2\\
&\geq & -2 t (\vp)^2(t),\quad t\in [a_{\alpha},1].
\end{eqnarray*}
Hence, we have
$$K_\alpha(x)=K_\alpha(x)-K_\alpha(1) \leq \int_x^1 2 t (\vp)^2(t)\, dt \leq 2 \int_0^1 (\vp)^2(t)\, dt=2,$$
which concludes the proof of the lemma.$\qquad \Box$\\

As a consequence of Lemma \ref{Sup_phi} and Lemma \ref{localestimate}, we obtain a bound for the eigenfunctions
$\vp,$ given by the following proposition.

\begin{proposition}\label{Sup2_phi}
Let $c>0,$ $ \alpha > -1/2$ be two real numbers. If  $c^2>\alpha^2-\frac{1}{4},$ and $\chi_{n,\alpha}(c) >c^2+\alpha^2-\frac{1}{4},$ then we have
\begin{equation}\label{Ineq2}
\sup_{x\in [a_{\alpha},1]}|\vp(x)| \leq 3 \sqrt{\frac{3}{2}} \left(\chi_{n,\alpha}(c)\right)^{1/2}.
\end{equation}
Here, $a_{\alpha}$ is as given by \eqref{Ineq2}.
\end{proposition}

\noindent
{\bf Proof:} Without loss of generality, we may assume that $\vp(1)>0.$  By integrating \eqref{differ_operator2} over the interval $[x,1],$ with $x\in J=[a_{\alpha},1),$ one gets
\begin{equation}
\label{deriv_phi}
(\vp)'(x)=\frac{\chi_{n,\alpha}(c)}{1-x^2}\int_x^1 \left(1- q t^2-\frac{\alpha^2-1/4}{\chi_{n,\alpha}(c)\, t^2}\right) \vp(t) \, dt,\quad q=c^2/\chi_{n,\alpha}(c).
\end{equation}
Let $G_{\alpha}$ be the function defined on $J$ by $$G_{\alpha}(t) =1- q t^2-\frac{\alpha^2-1/4}{\chi_{n,\alpha}(c) \, t^2}.$$ It can be easily checked that if $c^2>\alpha^2-\frac{1}{4},$ then
$G_{\alpha}$ is decreasing and positive in $J.$ Hence, by using \eqref{Ineq2} and \eqref{deriv_phi}, one gets
\begin{eqnarray*}
|(\vp)'(x)| &\leq & \frac{\chi_{n,\alpha}(c)}{1-x^2}G_{\alpha}(x) \vp(1) (1-x) =  \frac{\chi_{n,\alpha}}{1+x} G_{\alpha}(x) \vp(1).
\end{eqnarray*}
Consequently, we have,
\begin{equation}
\label{Ineqq1}
|\vp(1)-\vp(x)|\leq \frac{\chi_{n,\alpha}(c)}{1+x}  G_{\alpha}(x) (1-x) \vp(1).
\end{equation}
In a similar manner as it is  done in \cite{Bonami-Karoui1}, let $x_n\in J$ with
\begin{equation}
\label{Ineqq2}
G_{\alpha}(x_n) \frac{1-x_n}{1+x_n}=\frac{a}{\chi_{n,\alpha}(c)},
\end{equation}
 where $a>0$ is a constant to be fixed later on. By combining \eqref{Ineqq1} and \eqref{Ineqq2} and by using 
 the result of lemma 1, one gets
 \begin{equation}
 \label{Ineqq3}
\vp(1)\leq \frac{1}{1-a} |\vp(x_n)|\leq \frac{1}{1-a} \frac{\sqrt{2}}{\sqrt{1-x_n^2}}.
\end{equation}
On the other hand, since for any $x\in J,$ we have ${\displaystyle \frac{G_{\alpha}(x)}{1+x} \leq  1,}$ then from
\eqref{Ineqq2}, we have
$$\frac{1}{\sqrt{1-x_n^2}} \leq \frac{1}{\sqrt{1-x_n}} \leq \sqrt{\frac{\chi_{n,\alpha}}{a}}.$$
By combining the previous two inequalities, one gets
$$\vp(1)\leq \frac{1}{a^{1/2} (1-a)} \left(\chi_{n,\alpha}(c)\right)^{1/2}. $$ To conclude the proof, it suffices to note that the
minimum of the quantity $\frac{1}{a^{1/2} (1-a)}$ is obtained for $a= 1/3.$ $\qquad \Box$

To extend the previous result to the case where $\alpha >\frac{1}{2},$ and the interval $[a_{\alpha}, 1]$ is substituted
with the whole interval $[0,1],$ we first need to locate the first positive zero 
of $\vp.$ For this purpose, we use  the following Sturm-Liouville comparison theorem, that compares the zeros of the eigenfunctions of two
second order differential operators, see for example [\cite{Sturm-Liouville Theory Past and Present}, page 4]

\begin{theorem}[{\bf Sturm Comparison Theorem}] Let $p_i, r_i, i=1,2$ be two real continuous functions on the interval $[a,b]$ and let
$$\left(p_1(x) u'\right)'+r_1(x) u=0,\qquad  \left(p_2(x) v'\right)'+r_2(x) v=0,$$
be two ODE with $0<p_2(x)\leq p_1(x)$ and $r_1(x) \leq r_2(x).$ Then between any two zeros of $u,$ there exists a zero of $v.$
\end{theorem}

The following proposition gives a location of the first zero of $\vp,$ where $\alpha >1/2.$

\begin{proposition}\label{zero_location}
Let $c>0, \alpha >1/2$ be two real numbers. Let $ x_{1,n}$ be the first positive zero of $\vp.$ Then for any 
integer $n$ satisfying  $\chi_{n,\alpha}(c)\geq 2 c \sqrt{\alpha^2-1/4},$ we have
\begin{equation}\label{firstzero}
a_{\alpha}^2=\sqrt{\frac{\alpha^2-\frac{1}{4}}{\chi_{n,\alpha}(c)}}\leq  x_{1,n}\leq\frac{\pi+\frac{\pi}{2}\alpha -\frac{3}{4}}{\sqrt{\chi_{n,\alpha}(c)-(\alpha^2-1/4)}}=b_{\alpha}    .
\end{equation}
\end{proposition}

\noindent
{\bf Proof:} To prove the previous lower bound, we first note that $(\vp)'(0)=0$ whenever $\alpha >1/2.$ Moreover, from the equality,
\begin{equation}
\label{Diff_Equ2}
\frac{d}{d x} \left[(1-x^2) (\vp)'(x)\right]= \left(-\chi_{n,\alpha}(c)+c^2 x^2 +\frac{\alpha^2-1/4}{x^2}\right)\vp(x),
\end{equation}
one concludes that $\vp(x)$ and $(\vp)'(x)$ have the same positive  sign around $x=0,$ as long as the quantity
${\displaystyle -\chi_{n,\alpha}(c)+c^2 x^2 +\frac{\alpha^2-1/4}{x^2}\geq 0.}$ Straightforward computations show that this is the case when ${\displaystyle 0< x \leq r_1,}$ with ${\displaystyle r_1^2 =\frac{\chi_{n,\alpha}(c)}{2 c^2}\left(1-\sqrt{1-4c^2(\alpha^2-1/4)/(\chi_{n,\alpha}(c))^2}\right).}$ Consequently, we have
$$ x_{1,n} \geq r_1 \geq \sqrt{\frac{\alpha^2-\frac{1}{4}}{\chi_{n,\alpha}(c)}}.$$ 
To prove the upper bound in \eqref{firstzero}, we use the change of function 
\begin{equation}
\label{change_function}
U=(1-x^2)^{1/2}\vp
\end{equation}
that transforms the differential equation \eqref{differ_operator2} to 
the following equation for $U$, which has the same zeros as $\varphi$ on $(0, 1),$
\begin{equation}\label{modified}
U''+\left((1-x^2)^{-2}+\frac{\chi_{n,\alpha}(c)-c^2 x^2}{1-x^2}+\frac{\frac{1}{4}-\alpha^2}{x^2(1-x^2)}\right)U=0,\quad x\in (0,1).
\end{equation}
Since $\chi_{n,\alpha}(c) \geq c^2+\alpha^2-\frac{1}{4},$ then we have $-c^2 x^2 \geq -\chi_{n,\alpha}(c) x^2 +(\alpha^2-1/4)x^2.$ Consequently, we have  
\begin{eqnarray*}
(1-x^2)^{-2}+\frac{\chi_{n,\alpha}(c)-c^2 x^2}{1-x^2}+\frac{\frac{1}{4}-\alpha^2}{x^2(1-x^2)}&\geq & \frac{\chi_{n,\alpha}(c)-c^2 x^2}{1-x^2}+\frac{\frac{1}{4}-\alpha^2}{x^2(1-x^2)}\\
&\geq & \chi_{n,\alpha}(c) +\frac{\alpha^2-1/4}{1-x^2}\frac{x^4-1}{x^2}\\
&\geq &\chi_{n,\alpha}(c)-(\alpha^2-1/4)-\frac{\alpha^2-1/4}{x^2}.
\end{eqnarray*}
Then, we use  Sturm Comparison theorem to conclude that the first positive zero of $U$ or of $\vp$ lies before the second
zero of the bounded solution of the  differential equation,
\begin{equation}\label{Modified2}
V'' +\left( \big(\chi_{n,\alpha}(c)-(\alpha^2-1/4)\big)-\frac{\alpha^2-1/4}{x^2}\right) V =0,\quad x\in (0,1).
\end{equation}
It is well known that the bounded solution of the previous differential equation is given by
\begin{equation}
\label{functionV}
V(x)= \sqrt{x}\,  J_{\alpha}\left(\sqrt{\chi_{n,\alpha}(c)-(\alpha^2-1/4)} x\right).
\end{equation}
Note that since $x=0$ is a first zero of $V$ and since from \cite{Breen}, see also \cite{Elbert}, an upper bound of $j_{\alpha,k},$ 
the $k-$th positive zero of the Bessel function $J_{\alpha}(\cdot)$ is given by 
\begin{equation}\label{Besselzero}
j_{\alpha,k} < k \pi +\frac{\pi}{2}\alpha-\frac{0.965}{4} \pi < k \pi +\frac{\pi}{2} \alpha-\frac{3}{4}.
\end{equation}
Consequently, by using the Sturm comparison theorem applied to the equations \eqref{modified} and \eqref{Modified2},
one concludes that the first positive zero of $U$ or of $\vp$ lies before $$\frac{j_{\alpha,1}}{\sqrt{\chi_{n,\alpha}(c)-(\alpha^2-1/4)}} \leq \frac{\pi +\frac{\pi}{2}\alpha-\frac{3}{4}}{\sqrt{\chi_{n,\alpha}(c)-(\alpha^2-1/4)}},$$ which  concludes the proof of the proposition.\\

By using the results of proposition 1 and proposition 2, we get the following theorem that provides us with a bound
for $|\vp(x)|, x\in [0,1]$ which is valid for any $\alpha > -1/2.$

\begin{theorem} let $c>0$ and $\alpha > -1/2,$ be such that $c^2>\alpha^2-1/4.$ Then, for any  positive integer $n$ with 
$\chi_{n,\alpha}(c)\geq  c^2+\alpha^2-1/4$  and ${\displaystyle \frac{\sqrt{\chi_{n,\alpha}(c)} }{1-b_{\alpha}}}b_{\alpha}^{3/2}\leq 3\sqrt{3/2},$ where $b_{\alpha}$ is given by \eqref{firstzero},
 we have 
\begin{equation}\label{bound2}
\sup_{x\in [0,1]}|\vp(x)|\leq  3 \sqrt{\frac{3}{2}} \left(\chi_{n,\alpha}(c)\right)^{1/2}.
\end{equation}
\end{theorem}

\noindent
{\bf Proof:} We first recall that if $\alpha^2\leq \frac{1}{4},$ then $a_{\alpha}=0$ and the inequality \eqref{bound2} follows from proposition 1. Hence, it suffices to consider the case where $\alpha > 1/2.$ For this purpose, we use Butlewski's 
theorem, regarding the behaviour of the local extrema of the solution of a second order differential equation, see for example [\cite{Andrews}, p. 238]. More precisely, if $\phi $ is a solution of the  differential equation
\begin{equation}\label{Eq5.2.1}
\frac{d}{d t} \left(p(t)y'(t)\right)+q(t) y(t)=0,\quad t\in (a,b), 
\end{equation}
where $p(t)$ and $q(t)$ are two positive functions  belonging to $C'(a,b)$, then the local 
maxima of $\mid \phi \mid $ is increasing or decreasing, according to the condition that $p(t)q(t)$
is decreasing or increasing. In our case, we have 
$$p(t)= (1-t^2),\quad q(t)= q_n(t) = \chi_{n,\alpha}(c)-c^2 t^2-\frac{\alpha^2-1/4}{t^2},\quad t\in (0,1).$$
Since 
$$ \frac{d}{d t} (p(t) q_n(t))= (8 c^2 t^6- 4(\chi_{n,\alpha}(c)+c^2) t^4+4\alpha^2 -1)/(4t^3),$$
and  since $\chi_{n,\alpha}(c)\geq 2\alpha^2-1/2,$ then one  can easily check that there exists a unique real number
$t_{\alpha,n}\in \left[\sqrt{\frac{\alpha^2-1/4}{\chi_{n,\alpha}(c)}},\frac{1}{2}\right],$ so that the function 
$p(t) q_n(t)$ is increasing in $(0,t_{\alpha,n} )$ and decreasing in $(t_{\alpha,n}, 1).$ Hence, from Butlewski's theorem,
the local maxima of $|\vp|$ are decreasing in $(0,t_{\alpha,n} )$ and increasing in $(t_{\alpha,n}, 1).$ 
From the proof of proposition \ref{zero_location}, we know that the first zero of 
$(\vp)',$ denoted by $x'_{1,n}$ is located in $I_{\alpha}=[a_{\alpha}^2, b_{\alpha}],$ where $a_{\alpha}, b_{\alpha}$ are given by \eqref{firstzero}. Hence, by integrating \eqref{Diff_Equ2} over the interval $[x,x'_{1,n}],$ where $x\in I_{\alpha}$ and then using H\"older's inequality,  one gets 
\begin{eqnarray*}
(\vp)'(x) &=& \frac{-1}{1-x^2}\int_x^{x'_{1,n}} \left(-\chi_{n,\alpha}(c)+c^2 t^2 +\frac{\alpha^2-1/4}{t^2}\right)
\vp(t)\, dt,\quad a_{\alpha}^2\leq x\leq b_{\alpha}.
\end{eqnarray*}
On the other hand, from the expression of $(\vp(x))'',$ one can easily check that this later is positive whenever $0< x \leq a_{\alpha}^2.$ Consequently, we  have
$$|(\vp)'(x)| \leq \frac{\chi_{n,\alpha}(c)}{1-x^2} \int_x^{x'_{1,n}} \vp(t)\, dt,\quad 0< x < b_{\alpha}.$$
By using the expression of $b_{\alpha}$ as well as the conditions on $\chi_{n,\alpha}(c),$ together with H\"older's inequality applied to the above integral, one gets
$$|(\vp)'(x)| \leq \frac{\sqrt{b_{\alpha}}}{1-b_{\alpha}} \chi_{n,\alpha}(c).$$
Finally, since $\vp(0)=0$ and since $x'_{1,n}< b_{\alpha},$ then we have $$
|\vp(x'_{1,n})| < \frac{b_{\alpha}^{3/2}}{1-b_{\alpha}} \chi_{n,\alpha}(c) \leq 3 \sqrt{3/2}\sqrt{\chi_{n,\alpha}(c)}.$$
Finally, from the previous analysis, we have 
$$\sup_{x\in [0,1]}|\vp(x)| \leq \max\left(|\vp(x'_{1,n})|, |\vp(1)|\right)\leq 3 \sqrt{3/2}\sqrt{\chi_{n,\alpha}(c)},$$
which concludes the proof of the theorem.

The following theorem tells us that $\chi_{n,\alpha}(c)$ passes through $c^2+\alpha^2-\frac{1}{4}$ for
$n$ around $\frac{\sqrt{c^2+\alpha^2-1/4}}{\pi}.$

\begin{theorem} Consider two real numbers $c>0,$ $\alpha > -\frac{1}{2},$ with $c^2\geq \frac{1}{4}-\alpha^2,$
 then,\\
$(a)$ For any positive integer $n < \frac{c}{\pi}-\frac{\alpha}{2},$ we have  $\chi_{n,\alpha}(c)< c^2+\alpha^2-\frac{1}{4}.$\\
$(b)$ For any integer $n > \frac{\sqrt{c^2+\alpha^2-1/4}}{\pi} +\frac{5}{3},$ we have $\chi_{n,\alpha}(c)> c^2+\alpha^2-\frac{1}{4}.$
\end{theorem}

\begin{proof} To alleviate notation, we let $\varphi,$ $\chi$ denote the eigenfunction $\vp$ and its associated eigenvalue
$\chi_{n,\alpha}(c),$ respectively.
We want to prove that solutions on $(0, 1)$ of the differential equation
\begin{equation}\label{initial}
((1-x^2)\varphi')'+\left(\chi-c^2x^2+\frac{\frac{1}{4}-\alpha^2}{x^2}\right)\varphi=(P_1(x)\varphi')'+r_1(x) \varphi=0,
\end{equation}
have at least $\frac{\sqrt{\chi-(\alpha^2-1/4)}}{\pi}-\frac{\alpha}{2}$ zeros. As it is done in the proof of proposition 2, the change of function $U=(1-x^2)^{1/2}\varphi$ leads to the equation for $U$, given by \eqref{modified}.
Since $\chi \geq c^2+\alpha^2-\frac{1}{4},$ then from the Sturm comparison theorem, the number of zeros of $\vp$ is bounded below by the number of zeros of the function $V(\cdot),$ given by \eqref{functionV}. Since a bound of the $k-$th zero of the Bessel function $J_{\alpha}(\cdot)$ is given by \eqref{Besselzero}, then $n,$ the number of zeros of $\vp$ is bounded below by ${\displaystyle \left[\frac{\sqrt{\chi_{n,\alpha}(c)-(\alpha^2-1/4)}}{\pi}-\frac{\alpha}{2}\right]}.$ Finally, to conclude the proof of (a), it suffices to note that  $\chi \geq c^2+\alpha^2-\frac{1}{4}$ and use the previous bound below of the number of zeros $n.$\\

Next to prove (b), we divide the interval $(0, 1)$ into the two subintervals $(0,1-\eta),$ $[1-\eta,1)$ with
$\eta\in (0,\frac{1}{2})$ to be fixed later on.
We first bound the number $n_\eta$ of zeros of $\varphi$ or of $U$, in the interval $(0, 1-\eta)$. Since for
$0<x<1-\eta,$ we have $(1-x^2)^{-2}\leq \eta^{-2},$ then by using the Sturm-Liouville comparison theorem applied to
\eqref{initial} and the differential equation
\begin{equation}\label{modified2}
V''(x)+(\chi+\eta^{-2}) V(x)=0,\quad x\in (0,1-\eta),
\end{equation}
one concludes that
\begin{equation}\label{number1}
n_{\eta}\leq \frac{(1-\eta) \sqrt{\chi+\eta^{-2}}}\pi+1.
\end{equation}
It remains to find a bound for $n'_{\eta}=n-n_{\eta}.$ We now compare  the equation \eqref{initial} with
an appropriate second order differential equation on the interval $[1-\eta,1).$ We may assume that $\eta\leq 1-\sqrt{\frac{5}{6}}.$  Since in this last interval, we have $P_1(x)=(1+x)(1-x)\geq (2-\eta)(1-x)$ and since  $\chi\leq c^2+\alpha^2-\frac{1}{4}$ and $\frac{1+x^2}{x^2}\leq \frac{11}{5},$ then we have
\begin{eqnarray*}
r_1(x) &=&\chi-c^2 x^2 +\frac{\frac{1}{4}-\alpha^2}{x^2}\leq \chi(1-x^2)+\left(\frac{1}{4}-\alpha^2\right)\left(\frac{1}{x^2}-x^2\right)\\
&\leq&\chi (1-x^2)+(1/4-\alpha^2)(1-x^2) \frac{1+x^2}{x^2}\\
&\leq &(\chi+11/20)(1-x^2)\leq 2(\chi+11/20) (1-x)=r_2(x).
\end{eqnarray*}
Hence, we use the Sturm-Liouville comparison theorem applied to the equation \eqref{initial} and the following
equation
\begin{equation}\label{modified3}
(2-\eta)((1-x)U')'+2(\chi+11/20)(1-x)U=0,\quad x\in (1-\eta,1).
\end{equation}
The previous equation is rewritten as
$$ U'' -\frac{U'}{1-x}+\frac{2(\chi+11/20)}{2-\eta} U=0,\quad x\in (1-\eta,1).$$
If we let $v(x)=U(1-x)$ and take $t=1-x$ as a new variable,  then the previous equation is reduced
to the Bessel equation with solution
$v(t)=J_0(b t)$ on $(0, \eta)$, with $b^2=\frac{2\chi+11/10}{2-\eta}.$ Moreover, since from [\cite{Watson}, p.489],
 the $m-$th zeros of $J_0(x)$ lies in the interval $\left( (m+\frac{3}{4})\pi, (m+\frac{7}{8})\pi\right),$ then $v(t)$ has at most
 ${\displaystyle \left[\sqrt{\frac{2(\chi+11/20)}{2-\eta}}\frac{\eta}{\pi}-\frac{3}{4}\right]}$ zeros in $(0,\eta).$
 By using Sturm comparison theorem, one concludes that $n'_{\eta},$ the number of zeros of $\varphi$ in $(1-\eta,1)$ is bounded as follows,
 $$n'_{\eta}\leq \left[\sqrt{\frac{2(\chi+11/20)}{2-\eta}}\frac{\eta}{\pi}-\frac{3}{4}\right]+1\leq \frac{1}{\pi}\eta\sqrt{\frac{2(\chi+11/20)}{2-\eta}}+\frac{1}{4}.$$
 Straightforward manipulations show that
 \begin{eqnarray*}
 n= n_{\eta}+n'_{\eta}&\leq&\frac{1-\eta}{\pi}\sqrt{\chi}(1+\frac{1}{2}\eta^{-2}\chi^{-1})+1+\frac{\eta}{\pi}\sqrt{\frac{2}{2-\eta}(\chi+11/20)}+1/4\\
 &\leq & \frac{\sqrt{\chi}}{\pi}+\frac{1}{2\pi}\eta^{-2}\chi^{-1/2}+5/4+\frac{\eta}{\pi}\left(\frac{11/10+\eta \chi}{2\sqrt{\chi}}\right)
 \end{eqnarray*}
since $n\geq 1,$ then from \eqref{bounds_chi}, we have  $\chi\geq 6.$ Moreover, by choosing $\eta=\chi^{-\frac{1}{4}},$ one gets 
 $$\displaystyle\ n= n_{\eta}+n'_{\eta}\leq \frac{\sqrt{\chi}}{\pi}+ \frac{1+\frac{11}{20}6^{-3/4}}{\pi}+\frac{5}{4} \leq \frac{\sqrt{\chi}}{\pi}+\frac{5}{3},$$
 that is ${\displaystyle \sqrt{\chi_{n,\alpha}(c)} \geq \pi (n-5/3),}$ which allows us to conclude for (b).
  \end{proof}

\section{Eigenvalues behaviour and decay of the finite Hankel transform operator}

In this paragraph, we prove an important property of the eigenvalues $\lambda_{n,\alpha}(c),$ that is for fixed
integer $n\geq 0$ and real numbers  $c>0,$   $\alpha > \alpha' >-1/2,$ we have $\lambda_{n,\alpha'}(c) < \lambda_{n,\alpha}(c).$
To prove this result, we need the following Paley-Wiener theorem for the Hankel transform, given by J. L. Griffith in  \cite{Griffith}.

\begin{theorem}[\cite{Griffith}]
Let ${\displaystyle \alpha >  -\frac{1}{2}}$ and $p, q >0$ with ${\displaystyle \frac{1}{p}+\frac{1}{q}=1.}$ Let $f$ be an even function of exponential type 1. If $1< p \leq 2$ and ${\displaystyle t^{\alpha +1/2} f(t) \in L^p(0,+\infty)},$ then
$f$ can be represented by
$$f(z)= \int_0^1 \left(xz\right)^{-\alpha} J_{\alpha}(xz) \phi(x) \, dx,\quad z\in \mathbb C,$$
with $x^{-\alpha-1/2} \phi(x) \in L^q(0,1).$ Conversely, if $f$ has this representation and $x^{-\alpha-1/2} \phi(x) \in L^p(0,1),$ $ 1< p \leq 2,$ then $f$ is an even entire function of exponential type $1$ such that $t^{\alpha+1/2} f(t) \in L^q (0,\infty).$
\end{theorem}

By using the previous theorem, we prove the following lemma that compares two Paley-Wiener spaces for Hankel band-limited functions. We should mention that the previous theorem is still valid if the the interval $(0,1)$ is substituted with the interval $(0,c).$

\begin{lemma}
\label{lem1}
let $\alpha\geq  \alpha' > -\frac{1}{2}$  be two real numbers,  then the Hankel Paley-Wiener spaces  $B_c^\alpha $ and $B_c^{\alpha'}$ satisfy the following inclusion relation,
\begin{equation}\label{Paley_inclusion}
\mathcal B_c^\alpha\subset x^{\alpha-\alpha'} \cdot   \mathcal B_c^{\alpha'},\quad \alpha>\alpha'.
\end{equation}
Here, $\mathcal B_c^\alpha$ is as given by \eqref{Paley_space}.
\end{lemma}

\begin{proof}
Since $f\in \mathcal B_c^\alpha,$ then for  $x\ge 0,$ we have
\begin{eqnarray*}
  f(x)&=&\int_0^c\sqrt{xy}J_\alpha(xy) \mathcal H^\alpha(f)(y)dy
      = x^{\alpha+\frac{1}{2}}\int_0^c(xy)^{-\alpha}J_\alpha(xy)y^{\alpha+\frac{1}{2}} \mathcal H^\alpha(f)(y)dy.
\end{eqnarray*}
It follows that
$$x^{-\alpha-\frac{1}{2}}f(x)=\int_0^c(xy)^{-\alpha}J_\alpha(xy)y^{\alpha+\frac{1}{2}} \mathcal H^\alpha(f)(y)dy.$$
Let $\phi(y)=y^{\alpha+\frac{1}{2}} \mathcal H^\alpha(f)(y),$ then $y^{-\alpha-\frac{1}{2}}\phi(y)\in L^2[0,c]$. By using the previous Griffith's theorem with $p=q=2,$ one concludes that the function
$g= x^{-\alpha-\frac{1}{2}} f$ is an even entire function of exponential type 1. Moreover, since
$f=x^{\alpha+\frac{1}{2}} g\in L^2(0,+\infty)$ and since $\alpha>\alpha'> -\frac{1}{2},$ then
we have
\begin{equation}\label{1}
    x^{\alpha'+\frac{1}{2}}g \in L^2(0,+\infty).
\end{equation}
Again by using  Griffith's theorem, one concludes that  there exists a function $\varphi$ such that $x^{-\alpha'-\frac{1}{2}}\varphi\in L^2[0,c]$ and
$$g(x)=\int_0^c(xy)^{-\alpha'}J_{\alpha'}(xy)\varphi(y)dy.$$
Hence $$x^{\alpha'+1/2}g(x)=\int_0^c\sqrt{xy}J_{\alpha'}(xy)y^{-\alpha'-\frac{1}{2}}\varphi(y)dy.$$
It follows from (\ref{1}) that $x^{\alpha'+1/2} g =x^{\alpha'-\alpha}f\in L^2(0,+\infty)$ and $\displaystyle \mathcal H^\alpha(x^{\alpha'-\alpha}f)=x^{-\alpha'-\frac{1}{2}}\varphi 1_{[0,c]}.$ That is
 $x^{\alpha'-\alpha}f\in \mathcal B_c^{\alpha'}$ and  $f\in x^{\alpha-\alpha'} \mathcal B_c^{\alpha'}.$
\end{proof}

By using the previous lemma, we show that for a fixed integer $n\geq 0,$ the eigenvalues $\lambda_{n,\alpha}(c)$ is decreasing with respect to the parameter $\alpha >-\frac{1}{2}.$ This unexpected  result is one of the main results of this work and it is given by the following theorem.

\begin{theorem}\label{monotony}
Let $\big(\lambda_{n,\alpha}(c)\big)_{n\ge0}$ be the sequence of the eigenvalues of the operator
$\mathcal Q_c^{\alpha}= c \, \mathcal H_c^{\alpha}\mathcal H_c^{\alpha},$ then for any integer $n\geq 0,$ we have
\begin{equation}\label{Monotony}
\lambda_{n,\alpha}(c) \leq \lambda_{n,\alpha'}(c),\quad \forall\, \alpha \geq \alpha' > -\frac{1}{2}.
\end{equation}
\end{theorem}

\begin{proof}
We first recall that if $A$ is a self-adjoint compact operator on a Hilbert space $H,$ with positive eigenvalues $(\lambda_n)_n$ arranged in decreasing order, then by Min-Max theorem, we have
$${\displaystyle \lambda_k= \max_{S_k} \min_{x\in S_k} \frac{< Ax,x>}{\|x\|^2},}$$
where $S_k$ is a subspace of $H$ of dimension $k.$ In the special case where $H=\mathcal B_c^{\alpha},$  $A= \mathcal Q_c^{\alpha}$ and by using the discussion given in section 2, that  relates the energy maximization problem to the eigenvalues $\lambda_{n,\alpha}(c),$ one concludes that
\begin{equation*}
\lambda_{n,\alpha}(c)=\left\{
\begin{gathered}
\max_{f\in \mathcal B_c^{\alpha}}\frac{||f||^2_{L^2[0,1]}}{||f||^2_{L^2(0,+\infty)}},\mbox{ if } n=0\\
\max_{S_n\subset \mathcal B_c^{\alpha}}\min_{f\in S_n}\frac{||f||^2_{L^2[0,1]}}{||f||^2_{L^2(0,+\infty)}},\mbox{ if } n\ge1,
\end{gathered}\right.
\end{equation*}
where the $S_n$ are subspaces of $\mathcal B_c^{\alpha}$ of dimensions $n$.
Next, let $\alpha>{\alpha'}>-\frac{1}{2},$ then by using Lemma \ref{lem1}, we get $$\displaystyle\lambda_{0,\alpha}(c)\le\max_{f\in x^{\alpha-\alpha'} \mathcal B_c^{\alpha'}}\frac{||f||^2_{L^2[0,1]}}{||f||^2_{L^2(0,+\infty)}}=\displaystyle\max_{f\in \mathcal B_c^{\alpha'}}\frac{||x^{\alpha-\alpha'}f||^2_{L^2[0,1]}}{||x^{\alpha-\alpha'}f||^2_{L^2(0,+\infty)}}.$$
On the other hand, for $f\in \mathcal B_c^{\alpha'},$ we have $$\frac{||x^{\alpha-\alpha'}f||^2_{L^2(0,+\infty)}}{||x^{\alpha-\alpha'}f||^2_{L^2[0,1]}}=1+\frac{||x^{\alpha-\alpha'}f||^2_{L^2[1,+\infty)}}{||x^{\alpha-\alpha'}f||^2_{L^2[0,1]}}\ge 1+\frac{||f||^2_{L^2[1,+\infty)}}{||f||^2_{L^2[0,1]}}=\frac{||f||^2_{L^2(0,+\infty)}}{||f||^2_{L^2[0,1]}},$$
which implies that
\begin{equation}\label{2}
 \frac{||x^{\alpha-\alpha'}f||^2_{L^2[0,1]}}{||x^{\alpha-\alpha'}f||^2_{L^2[0,+\infty)}}\le \max_{f\in \mathcal B_c^{\alpha'}}\frac{||f||^2_{L^2[0,1]}}{||f||^2_{L^2[0,+\infty)}}.
 \end{equation}
That is  $$\lambda_{0,\alpha}(c)\le \lambda_{0,\alpha'}(c).$$
Similarly, for $n\ge 1,$ and by using Lemma \ref{lem1}, we get
\begin{eqnarray*}
\lambda_{n,\alpha}(c)&\le&\max_{S_n\subset x^{\alpha-\alpha'} \mathcal B_c^{\alpha'}}\min_{f\in S_n}\frac{||f||^2_{L^2[0,1]}}{||f||^2_{L^2[0,+\infty)}}\\
                     &\le&\max_{x^{\alpha'-\alpha}S_n\subset \mathcal B_c^{\alpha'}}\min_{f\in S_n}\frac{||f||^2_{L^2[0,1]}}{||f||^2_{L^2[0,+\infty)}}\\
                     &\le&\max_{x^{\alpha'-\alpha}S_n\subset \mathcal B_c^{\alpha'}}\min_{g\in x^{\alpha'-\alpha}S_n}\frac{||x^{\alpha-\alpha'}g||^2_{L^2[0,1]}}{||x^{\alpha-\alpha'}g||^2_{L^2[0,+\infty)}}\\
                     &\le&\max_{H_n\subset \mathcal B_c^{\alpha'}}\min_{g\in H_n}\frac{||x^{\alpha-\alpha'}g||^2_{L^2[0,1]}}{||x^{\alpha-\alpha'}g||^2_{L^2[0,+\infty)}}
\end{eqnarray*}
Hence, by (\ref{2}) we get $\displaystyle\lambda_{n,\alpha}(c)\le \max_{H_n\subset HB_c^{\alpha'}}\min_{g\in H_n}\frac{||g||^2_{L^2[0,1]}}{||g||^2_{L^2[0,+\infty[}}=\lambda_{n,\alpha'}(c),$ which completes the proof of the theorem.
\end{proof}

Note that in the special case where $\alpha=\frac{1}{2},$ we have ${\displaystyle J_{1/2}(x)=\sqrt{\frac{2}{\pi x}} \sin(x)}$ and the $\varphi_{n,c}^{1/2}$ are the solutions of the eigen-problem
\begin{equation}\label{Eigenproblem1}
\sqrt{\frac{2}{\pi}} \int_0^1 \sin( cxy)\varphi_{n,c}^{1/2}(y)\, dy =\mu_{n,1/2}(c) \varphi_{n,c}^{1/2}(x), \quad x\in [0,1].
\end{equation}
Moreover, it is well known that the solutions of the previous eigen-problem are given by the  classical prolate spheroidal wave functions of odd orders $\psi_{2n+1,c}.$ These PSWFs are solutions of the  integral equations,
\begin{equation}\label{Eigenproblem2}
2i \int_0^1 \sin( cxy)\psi_{2n+1,c}(y)\, dy =\mu_{2n+1}(c) \psi_{2n+1,c}(x),
\end{equation}
\begin{equation}
\int_{-1}^1 \frac{\sin c(x-y)}{\pi (x-y)} \psi_{2n+1,c}(y)\, dy= \lambda_{2n+1}(c) \psi_{2n+1,c}(x)  \quad x\in [0,1].
\end{equation}
From the previous three equalities, one gets the following identity relating the eigenvalues of $\mathcal Q_c^{1/2}$
to the eigenvalues associated to  the classical  PSWFs of odd orders,
\begin{equation}\label{identity_eigenvals}
\lambda_{n,1/2}(c)= c |\mu_{n,1/2}(c)|^2=\frac{c}{2\pi} |\mu_{n}(c)|^2 =\lambda_{2n+1}(c), n\geq 0.
\end{equation}
Note that unlike the eigenvalues $\lambda_{n,\alpha}(c),$ the behaviour and the sharp decay rate of  eigenvalues $\lambda_n(c)$ associated with the classical PSWFs, are well known in the literature, see for example \cite{Bonami-Karoui2, Jaming-Karoui-Spektor, Landau1, Slepian4}. In particular, it has been shown  in \cite{Bonami-Karoui2} that the sharp asymptotic  decay rate of the $(\lambda_{n}(c))$ is given by
${\displaystyle e^{-2n \log\left(\frac{4n}{ec}\right)}}.$ More precisely, for any real $0< a < \frac{4}{e},$ there exists a constant $M_a$ such that ${\displaystyle \lambda_{n,c} \leq e^{-2n \log\left(\frac{an}{c}\right)}},$ for
$n\geq c M_a.$ Moreover, for any real $b > \frac{4}{e},$ there exists a constant $M_b$ such that ${\displaystyle \lambda_{n,c} \geq e^{-2n \log\left(\frac{bn}{c}\right)}},$ for
$n\geq c M_b.$ By combining the monotonicity of the $\lambda_{n,\alpha}(c)$ with respect to the parameter $\alpha,$
the identity \eqref{identity_eigenvals} and the previous decay rate of the classical eigenvalues $\lambda_{n}(c),$
one gets the following corollary that provides us with a super-exponential decay rate of the $\lambda_{n,\alpha}(c).$

\begin{corollary}
Let $c>0$ and $\alpha \geq \frac{1}{2}$ be two positive real numbers. Then for any $0<a<\frac{8}{e},$ there exits
a constant $M_a$ such that
\begin{equation}\label{Decay_Eigenvalues}
\lambda_{n,\alpha}(c) \leq e^{-4 n \log\left(\frac{ a n}{c}\right)},\quad n\geq c M_a.
\end{equation}
\end{corollary}

Unfortunately and unlike the classical case, we don't have a precise asymptotic lower decay rate of the
$\lambda_{n,\alpha}(c).$ Nonetheless, the following proposition gives us a bound below for the asymptotic decay rate of the $\lambda_{n,\alpha}(c),$ with a similar type of the  super-exponential decay of the bound above.

\begin{proposition}
Let $c>0$ be a positive real number, then there exists a constant $\delta_0$ and a positive integer  $k_0$
such that for any integer ${\displaystyle n\geq \max\left(\frac{c}{2},\frac{c}{\pi}+k_0\right)}$ and $ \chi_{n,\alpha}(c) > \max\left(2\alpha^2-1/2, c^2(4\alpha^2-1)\right),$ $c^2/\chi_{n,\alpha}(c),$ we  have
\begin{equation}\label{bound1_lambda}
|\lambda_{n,\alpha}(c)| \geq \delta_0 e^{-  A (2n+\alpha +1) \log\left(\frac{\pi}{c}(n+k_0)\right)},
\end{equation}
for some positive constant $A.$
\end{proposition}

\noindent
{\bf Proof:} It is well known, see \cite{Slepian3} that $\mu_{n,\alpha}(c)$ satisfies the differential equation,
$$\frac{\partial \mu_{n,\alpha}(c)}{\partial c} = \frac{\mu_{n,\alpha}(c)}{2c } \left((\varphi^{(\alpha)}_{n,c}(1))^2 -1\right).$$
Here, we recall that $\varphi_{n,c}$ is normalized so that $\|\varphi_{n,c}\|_{L^2(0,1)}=1.$ It can be easily checked that in this case, $\lambda_{n,\alpha}(c)$ satisfies
\begin{equation}\label{lambda_n}
\frac{\partial }{\partial c}\left(\log(\lambda_{n,\alpha}(c))\right) = \frac{(\varphi^{(\alpha)}_{n,c}(1))^2}{c}.
\end{equation}
On the other hand, from \cite{Abreu}, there exists a positive  integer $k_0$ and a positive real number $\delta_0$ such that
$$\lambda_{\left[ \frac{c}{\pi}\right]-k_0,\alpha}(c) \geq \delta_0.$$
Since the $(\lambda_{n,\alpha}(c))_n$ are arranged in the decreasing order, then the previous inequality implies that
$\lambda_{n,\alpha}(c) \geq \delta_0$ for any integer $n\leq \left[ \frac{c}{\pi}\right]-k_0,$ or any $c\geq c_n=\pi (n+k_0).$
  Also, by using \eqref{lambda_n}, one gets
\begin{equation}\label{lower_bound}
\lambda_{n,\alpha}(c) = \lambda_{n,\alpha}(c_n) \exp\left(-\int_c^{c_n} \frac{(\varphi^{(\alpha)}_{n,\tau}(1))^2}{\tau}\right)\, d\tau \geq \delta_0
\exp\left(-\int_c^{c_n} \frac{(\varphi^{(\alpha)}_{n,\tau}(1))^2}{\tau}\right)\, d\tau.
\end{equation}
On the other hand, it has been shown in \cite{Karoui-Mehrzi} that for $\alpha \geq \frac{1}{2}$ and $ \chi_{n,\alpha}(c) > \max\left(2\alpha^2-1/2, c^2(4\alpha^2-1)\right),$ the  WKB uniform approximation of the $\vp$ is given by
\begin{equation}\label{Uniform_Approx}
\sup_{x\in [\gamma_n,1]}\left|\varphi^{(\alpha)}_{n,c}(x)- \frac{A_n \chi_n^{1/4}\sqrt{S_n(x)} J_0(\sqrt{\chi_n} S_n(x))}{(1-x^2)^{1/4} r_n(x)^{1/4}}\right|
\leq \frac{C_{q_n}}{\sqrt{\chi_n}},
\end{equation}
where $A_n$ is a normalization constant, $C_{q_n}$ is a constant depending only on $q_n=c^2/\chi_{n,\alpha}(c)<1,$
$\gamma_n=\sqrt{\frac{2\alpha^2-1/2}{\chi_{n,\alpha}(c)}},$ and
$$r_n(t)=1-q_n t^2-\frac{\alpha^2-1/4}{t^2 \chi_{n,\alpha}(c)},\quad
S_n(x)=\int_x^1 \sqrt{\frac{r_n(t)}{1-t^2}}\, dt.$$
Also, from \cite{Karoui-Mehrzi}, we know that in the neighbourhood of $x=1,$ the quantity  ${\displaystyle \frac{\sqrt{S_n(x)} J_0(\sqrt{\chi_n} S_n(x))}{(1-x^2)^{1/4} r_n(x)^{1/4}}}$ is bounded uniformly in  $n.$ Moreover, by using the same techniques as those used in \cite{Bonami-Karoui3} for the approximation of the normalization constant 
 appearing  in the WKB approximation of the classical PSWFs, one concludes that our normalization constants $A_n$ are also bounded uniformly in $n$  as soon as $q_n=c^2/\chi_{n,\alpha}(c) \leq \widetilde q <1.$ Consequently, if we also assume that $n\geq c/2,$
 then using the previous analysis together with the upper bound of $\chi_{n,\alpha}(c),$ given by \eqref{bounds_chi}, one concludes that there exists a constant $B$ such that
 \begin{equation}\label{bounds_varphi}
 (\vp(1))^2 \leq B \sqrt{\chi_{n,\alpha}(c)}\leq  \sqrt{2} B \cdot (2n+\alpha+1).
 \end{equation}
It is easy to see that the previous inequality is still valid for any $c_n=\pi (n+k_0)\leq \tau \leq c.$ Finally, by substituting  $c$ with $\tau$ in \eqref{bounds_varphi} and using \eqref{lower_bound}, one gets the desired
result \eqref{bound1_lambda}.

As a consequence of the previous proposition, we have the following corollary showing the super-exponential decay rate of the   $|\mu_{n,\alpha}|$    does not invalidate
an exponential decay of the expansion coefficients $(d_k^n)_k,$ given by \eqref{Eq1.9}.

\begin{corollary}
Under the hypotheses on the integer $n,$ given by the previous proposition, there exist two positive constants $A, M$ such that for any 
integer $k\geq n,$ we have
\begin{equation}
\label{decay_coeff}
|d_k^n|\leq \dfrac{M}{\sqrt{c \pi (2k+\alpha+/2)}} \exp\left({-(2k+\alpha+1)\log\left(\frac{4k+2\alpha+2}{e c}\right)+A n \log\left(\frac{\pi n}{c}\right)}\right).
\end{equation}
\end{corollary}

\noindent
{\bf Proof:}  We first note that the Bessel function satisfies the following bound, see for example \cite{Andrews},
${\displaystyle |J_{\alpha}(x)|\leq \dfrac{|x|^\alpha}{2^\alpha \Gamma(\alpha+1)} \,\,\, \forall \alpha > \dfrac{-1}{2}.}$
Here, $\Gamma$ denotes the Gamma function.
By combining  \eqref{Eq1.9} and the previous inequality, one obtains
\begin{eqnarray*}
|d_k^n| &\leq& \dfrac{\sqrt{2(2k+\alpha+1)}}{|\mu_{n,\alpha}(c)|}  \int_{0}^{1}{|\vp(y)|  \dfrac{|cy|^{2k+\alpha+1}}{2^{2k+\alpha+1} \Gamma(2k+\alpha+2) \sqrt{cy}} dy}\\
&\leq& \dfrac{\sqrt{2(2k+\alpha+1)}c^{2k+\alpha+\dfrac{1}{2}}}{|\mu_{n,\alpha}(c)|2^{2k+\alpha+1} \Gamma(2k+\alpha+2)}  \int_{0}^{1}{|\vp(y)|  y^{2k+\alpha+ \dfrac{1}{2}} dy}\\
&\leq& \dfrac{\sqrt{2(2k+\alpha+1)}c^{2k+\alpha+\dfrac{1}{2}}}{|\mu_{n,\alpha}(c)|2^{2k+\alpha+1} \Gamma{(2k+\alpha+2})(2k+\alpha+\dfrac{3}{2})}.
\end{eqnarray*}
The last inequality follows from the H\"older's inequality applied to the integral ${\displaystyle \int_{0}^{1}{|\vp(y)|  y^{2k+\alpha+ \dfrac{1}{2}} dy}.}$
On the other hand, it is well known that $\Gamma(s+1)\geq \sqrt{2\pi}  s^{s+\dfrac{1}{2}} \exp{(-s)}.$ Consequently, we have
\begin{equation}
|d_k^n|\leq\frac{\sqrt{2}}{|\mu_{n,\alpha}|}\dfrac{1}{\sqrt{c \pi (2k+\alpha+3/2)}} e^{-(2k+\alpha+1)\log\left(\frac{4k+2\alpha+2}{e c}\right)}.
\end{equation}
Finally, by combining the previous inequality and \eqref{bound1_lambda} and taking into account that
$\lambda_{n,\alpha}(c)= c |\mu_{n,\alpha}(c)|^2,$ one gets the desired inequality \eqref{decay_coeff}.

\section{Numerical Results}

In  this paragraph, we give some numerical examples that illustrate the various results of the previous sections. Moreover, we show that the eigenfunctions of the finite Hankel transform operator are well adapted for the approximation of Hankel- and almost Hankel Band-limited functions.\\

\noindent{\bf Example 1:} In this example, we illustrate one of the main results of this work, which is given by Theorem \ref{monotony}. That is the eigenvalues $\lambda_{n,\alpha}(c)$ are decreasing with respect to the parameter $\alpha.$ For this purpose, we have considered the values of $c=10 \pi$ and the four values of $\alpha=0, 1,2, 3.$ Then, we have used formula \eqref{Eigenvals1} and computed highly accurate approximation of the eigenvalues  $\lambda_{n,\alpha}(c)$ with $0\leq n \leq 40.$ In Figure 1(a), we have plot the graphs of the significant eigenvalues $\lambda_{n,\alpha}(c)$
with the various values of $n$ and $\alpha.$ In order to check that the decay with respect to the parameter $\alpha$ holds also for the very small eigenvalues, we have plot in Figure 1(b) the graphs of the $\log(\lambda_{n,\alpha}(c)).$
Note that the results given by the previous figures indicate what was expected by Theorem \ref{monotony},  that is
the $\lambda_{n,\alpha}(c)$ are decreasing with respect to the parameter $\alpha.$

\begin{figure}[h]\hspace*{0.05cm}
{\includegraphics[width=15.05cm,height=5.5cm]{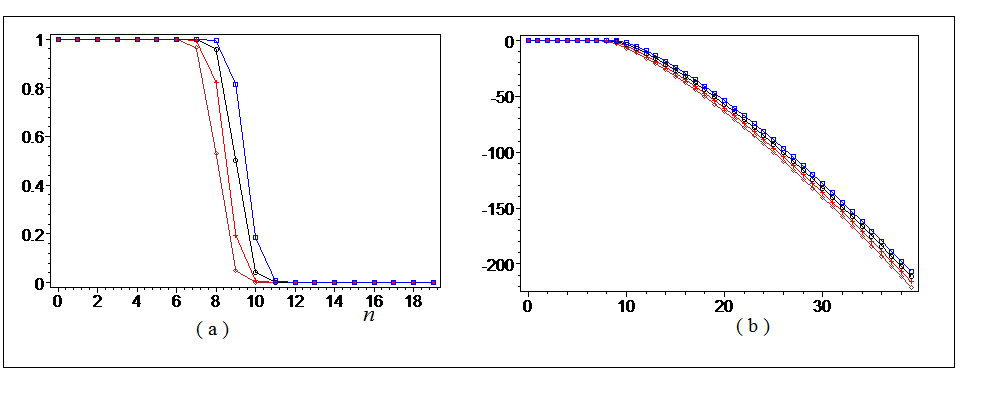}}
\vskip -0.5cm\hspace*{1cm} \caption{(a) Graph of the $\lambda_{n,\alpha}(c)$ for $c=10\pi,$ $\alpha=0$ (blue), 1(black), 2 (red), 3(brown),$\quad$ (b) same as (a) with the graphs of  $\log\left(\lambda_{n,\alpha}(c)\right).$}
\end{figure}

\vskip 0.25cm
\noindent{\bf Example 2:} In this example, we give some numerical tests that illustrate the super-exponential decay rate of the eigenvalues $\lambda_{n,\alpha}(c),$ given by Corollary 1. For this purpose, we have considered the value of $\alpha=1$ and the three different values of $c=5\pi, 10\pi, 15\pi,$  and computed highly accurate values of the eigenvalues $\lambda_{n,\alpha}(c),$ for ${\displaystyle   n\geq \frac{\sqrt{c^2+\alpha^2-1/4}}{\pi}+\frac{5}{3}.}$ By Theorem 2, these values of $n$ correspond to the case where $\chi_{n,\alpha}(c)\geq c^2+\alpha^2-\frac{1}{4}.$
As in the classical case, the critical value of $n=n_c=\frac{\sqrt{c^2+\alpha^2-1/4}}{\pi}$ corresponds to the beginning
of the plunge region of the eigenvalues $\lambda_{n,\alpha}(c).$ In Figure 2, we plot the graphs of the highly accurate values of the
$\log(\lambda_{n,\alpha}(c)),$ as well as  the graphs of $-4 n \log\left(\frac{ 8 n}{e c}\right),$ the logarithm of  the optimal theoretical super-exponential decay rate, as given by corollary 1. Note that for the different values of $c,$
the theoretical  asymptotic decay rate given by corollary 1 is very close to the actual decay rate.

\begin{figure}[h]\hspace*{3.05cm}
{\includegraphics[width=11cm,height=6.5cm]{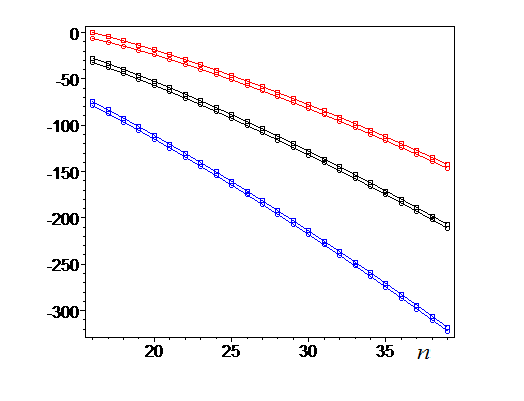}}
\vskip -0.5cm\hspace*{1cm} \caption{(a) Graph of the $\log(\lambda_{n,\alpha}(c))$ for $\alpha=1,$ $c=5\pi$ (blue circles), $c= 10 \pi$ (black circles), $c=15 \pi$ (red circles), versus the corresponding  graphs of       $-4 n \log\left(\frac{ 8 n}{e c}\right)$ (boxes).}
\end{figure}

\vskip 0.5cm
\noindent{\bf Example 3:} In this last example, we illustrate the quality of the spectral approximation of the Hankel band-limited and almost Hankel band-limited functions, by the orthogonal projection over $\mbox{Span}\{\vp , \,\, 0\leq n\leq N\}.$ Note that the concept of almost band-limited functions has been introduced in the framework of the classical Fourier transform
by Landau, see \cite{Landau1}. In a similar manner, the almost Hankel band-limited functions are defined as follows.

\begin{definition}
Let $\Omega$ be a measurable set of $\mathbb R_+$ and let $\epsilon_{\Omega}>0$ be a positive real number. A function
$f\in L^2(0,+\infty)$ is said to be $\epsilon_{\Omega}-$almost band-limited to $\Omega$ if
\begin{equation}
\| \mathcal H^{\alpha} f - \chi_{\Omega}\mathcal H^{\alpha} f\|_{L^2(0,+\infty)}\leq \epsilon_{\Omega}.
\end{equation}
Here $\chi_{\Omega}$ denotes the characteristic function of $\Omega.$
\end{definition}

In particular, since the $\vp$ are Hankel $c-$band-limited functions, then they are $0-$almost band-limited to $\Omega=[0,c].$ Next, for an integer $N\geq 1,$ let $S^{\alpha}_{N}(f)$ be the $N$-th partial sum of the  expansion of $f,$ in the basis $\{\vp,\,\, n\geq 0\},$ that is
\begin{equation}\label{eqqq6.1}
S^{\alpha}_N f (x) = \sum_{n=0}^N <f,\vp>_{L^2(0,1)}\vp(x),
 \end{equation}
 where $<\cdot,\cdot>$ denotes the usual inner product of $L^2(0,1).$ The quality of approximation  of the classical Fourier almost  $c-$band-limited functions, by the classical PSWFs has been given in \cite{Bonami-Karoui4}. Moreover, in \cite{Jaming-Karoui-Spektor}, this quality of approximation has been extended to the expansion with respect to some families of classical orthogonal polynomials. By straightforward modifications  of the techniques used in  \cite{Jaming-Karoui-Spektor}, one gets the following proposition that provides
us with the quality of approximation of almost Hankel $c-$band-limited function by the $\vp.$
\begin{proposition} If $f$ is an $L^2(0,+\infty)$ function that is
$\epsilon_{\Omega}-$ almost Hankel band-limited in $\Omega=[0, c],$ then
for any positive integer $N,$ we have
\begin{equation}\label{Approx1}
\left( \int_{0}^{+1}|f(t)- S^{\alpha}_N f(t)|^2 dt\right)^{1/2} \leq  \left(\epsilon_{\Omega}+\sqrt{\lambda_{N,\alpha}(c)}\right)\, \|f\|_{L^2(0,+\infty)}.
 \end{equation}
\end{proposition}

\begin{figure}[h]\hspace*{1.05cm}
{\includegraphics[width=15cm,height=4.5cm]{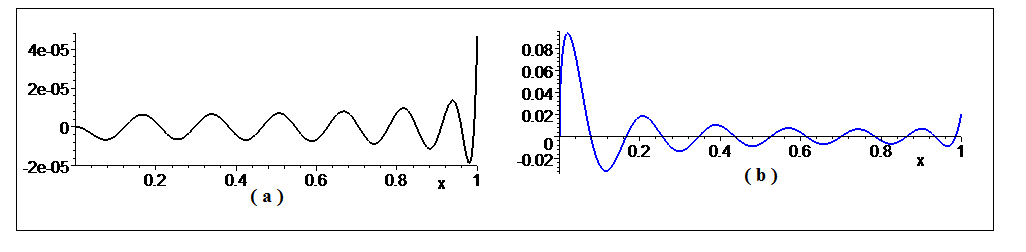}}
\vskip -0.5cm\hspace*{1cm} \caption{(a) Graph of the approximation error $f_1(x)-S^{\alpha_1}_N f_1$ for $\alpha_1=3/2,$ $c=10\pi,$ $N=11,$  (b) Graph of the approximation error $f_2(x)-S^{\alpha_2}_N f_2$ for $\alpha_2=1,$ $c=10\pi,$ $N=11.$ }
\end{figure}

To illustrate the previous spectral approximation result, we have considered the following Hankel and almost Hankel band-limited functions, given by
$$f_1(x)= \frac{J_{\alpha_1+1}(a x)}{\sqrt{x}},\,\alpha_1=\frac{3}{2},\, a=20;\qquad\qquad f_2(x)= x^{\alpha_2-1/2}\exp(-x),\,\, \alpha_2=1,$$ respectively. Note that the Hankel transforms of $f_1$ and $f_2$ are given by
$$\mathcal H^{\alpha_1}(f_1)(s)=a^{-\alpha_1-1}s^{\alpha_1+1/2} \mathbf 1_{[0,a]}(s),\qquad \mathcal H^{\alpha_2}(f_2)(s)= \frac{2^{\alpha_2}\Gamma(\alpha_2+1/2)}{\sqrt{\pi}}\frac{s^{\alpha_2+1/2}}{(1+s^2)^{\alpha_2+1/2}}.$$
Hence, $f_1\in B_c^{\alpha_1},$ for any $c\geq a.$ Moreover, straightforward computations show that 
$f_2$ is $\epsilon_{\Omega}-$concentrated on $\Omega=[0,c],$ with 
$$\epsilon_{\Omega}=\frac{2^{\alpha_2}\Gamma(\alpha_2+1/2)}{\sqrt{\pi}}\frac{\sqrt{1+2c^2}}{2 (1+c^2)^2}.$$
In the special case where $\alpha_2=1$ and $c=10\pi,$ we have $\epsilon_{\Omega}\approx 0.0225.$ For this last value of $c=10 \pi,$ we have computed the $N$-th partial sum $S^{\alpha_1}_N f_1$ and $S^{\alpha_2}_N f_2,$ with $N=11.$
The approximation errors $f_1(x)-S^{\alpha_1}_N f_1$ and $f_2(x)-S^{\alpha_2}_N f_2$ are given in Figure 3(a) and 3(b), respectively. Note that as predicted by the previous proposition, the first approximation error is proportional to 
$\sqrt{\lambda_{N,\alpha_1}(c)},$ and the second one is proportional to $\epsilon_{\Omega}.$

\end{document}